\documentclass[a4paper,12pt]{article}
\usepackage[margin=1.1in]{geometry}  

\usepackage{graphicx}              
\usepackage{amsmath}
\usepackage{amscd}               
\usepackage{amsfonts} 
\usepackage{amssymb}             
\usepackage{amsthm}                
\usepackage{mathrsfs}
\usepackage{amsbsy}
\usepackage{tikz}
\usetikzlibrary{decorations.pathmorphing}
\usepackage{hyperref}

\hypersetup{pdfstartview=}

\numberwithin{equation}{section}

\newtheorem{thm}{Theorem}[section]

\newtheorem{lem}[thm]{Lemma}
\newtheorem{prop}[thm]{Proposition}
\newtheorem{cor}[thm]{Corollary}

\newtheorem{assumption}[thm]{Assumption}
\DeclareMathOperator{\id}{id}

\makeatletter
\def\th@newremark{\th@remark\thm@headfont{\bfseries}}
\makeatletter

\theoremstyle{newremark}
\newtheorem{rmk}[thm]{Remark}

\newcommand{\NN}{\mathbb{N}}

\newcommand{\TT}{\mathbb{T}}

\newcommand{\bB}{\mathcal{B}}
\newcommand{\cC}{\mathcal{C}}
\newcommand{\dD}{\mathcal{D}}

\newcommand{\lL}{\mathcal{L}}
\newcommand{\mM}{\mathcal{M}}

\newcommand{\xX}{\mathcal{X}}

\newcommand{\E}{\mathbf{E}}

\newcommand{\R}{\mathbf{R}}
\newcommand{\T}{\mathbf{T}}

\newcommand{\eps}{\varepsilon}
\renewcommand{\d}{\partial}
\renewcommand{\div}{\text{div}}

\newcommand{\diam}{\text{diam}}

\definecolor{darkgreen}{rgb}{0.1,0.7,0.1}
\definecolor{darkred}{rgb}{0.7,0.1,0.1}
\definecolor{darkblue}{rgb}{0,0,0.7}
\def\scal#1{\langle#1\rangle}
\def\bigscal#1{\big\langle#1\big\rangle}
\def\Bigscal#1{\Big\langle#1\Big\rangle}

\colorlet{symbols}{blue!90!black}
\colorlet{testcolor}{green!60!black}

\def\drawx{\draw[-,solid] (-3pt,-3pt) -- (3pt,3pt);\draw[-,solid] (-3pt,3pt) -- (3pt,-3pt);}
\tikzset{
	root/.style={circle,fill=testcolor,inner sep=0pt, minimum size=2mm},
	dot/.style={circle,fill=black,inner sep=0pt, minimum size=1mm},
	var/.style={circle,fill=black!10,draw=black,inner sep=0pt, minimum size=
	2mm},
	dotred/.style={circle,fill=black!50,inner sep=0pt, minimum size=2mm},
	generic/.style={semithick,shorten >=1pt,shorten <=1pt},
	dist/.style={ultra thick,draw=testcolor,shorten >=1pt,shorten <=1pt},
	testfcn/.style={ultra thick,testcolor,shorten >=1pt,shorten <=1pt,<-},
	testfcnx/.style={ultra thick,testcolor,shorten >=1pt,shorten <=1pt,<-,
		postaction={decorate,decoration={markings,mark=at position 0.6 with {\drawx}}}},
	kepsilon/.style={semithick,shorten >=1pt,shorten <=1pt,densely dashed,->},
	kprimex/.style={semithick,shorten >=1pt,shorten <=1pt,densely dashed,->,
		postaction={decorate,decoration={markings,mark=at position 0.4 with {\drawx}}}},
	kernel/.style={semithick,shorten >=1pt,shorten <=1pt,->},
	multx/.style={shorten >=1pt,shorten <=1pt,
		postaction={decorate,decoration={markings,mark=at position 0.5 with {\drawx}}}},
	kernelx/.style={semithick,shorten >=1pt,shorten <=1pt,->,
		postaction={decorate,decoration={markings,mark=at position 0.4 with {\drawx}}}},
	kernel1/.style={->,semithick,shorten >=1pt,shorten <=1pt,postaction={decorate,decoration={markings,mark=at position 0.45 with {\draw[-] (0,-0.1) -- (0,0.1);}}}},
	kernel2/.style={->,semithick,shorten >=1pt,shorten <=1pt,postaction={decorate,decoration={markings,mark=at position 0.45 with {\draw[-] (0.05,-0.1) -- (0.05,0.1);\draw[-] (-0.05,-0.1) -- (-0.05,0.1);}}}},
	kernelBig/.style={semithick,shorten >=1pt,shorten <=1pt,decorate, decoration={zigzag,amplitude=1.5pt,segment length = 3pt,pre length=2pt,post length=2pt}},
	gepsilon/.style={dotted,semithick,shorten >=1pt,shorten <=1pt},
	renorm/.style={shape=circle,fill=white,inner sep=1pt},
	labl/.style={shape=rectangle,fill=white,inner sep=1pt},
	xi/.style={circle,fill=symbols!10,draw=symbols,inner sep=0pt,minimum size=1.2mm},
	xix/.style={crosscircle,fill=symbols!10,draw=symbols,inner sep=0pt,minimum size=1.2mm},
	xib/.style={circle,fill=symbols!10,draw=symbols,inner sep=0pt,minimum size=1.6mm},
	xibx/.style={crosscircle,fill=symbols!10,draw=symbols,inner sep=0pt,minimum size=1.6mm},
	not/.style={circle,fill=symbols,draw=symbols,inner sep=0pt,minimum size=0.5mm},
	>=stealth,
	}

\makeatletter
\def\DeclareSymbol#1#2#3{\expandafter\gdef\csname MH@symb@#1\endcsname{\tikz[baseline=#2,scale=0.15,draw=symbols]{#3}}\expandafter\gdef\csname MH@symb@#1s\endcsname{\scalebox{0.7}{\tikz[baseline=#2,scale=0.15,draw=symbols]{#3}}}}
\def\<#1>{\csname MH@symb@#1\endcsname}
\makeatother

\DeclareSymbol{Xi22}{0.5}{\draw (0,0) node[xi] {} -- (-1,1) node[not] {} -- (0,2) node[xi] {}; }

\begin{document}

\title{Periodic homogenisation for $P(\phi)_2$}

\author{
Yilin Chen\\
\textit{Peking University}\\
\and
Weijun Xu\\
\textit{Peking University}
}

\maketitle

\abstract{We consider the periodic homogenisation problem for dynamical $P(\phi)_2$, a toy model that combines both renormalisation in singular stochastic PDEs and homogenisation. Our result shows that the two limiting procedures commute in this case.}

\tableofcontents

\section{Introduction}

The aim of this article is to establish the well-posedness and then the $\eps \rightarrow 0$ limit of the solution $\phi_\eps$ to the equation
\begin{equation} \label{eq:main}
	\d_t \phi_\eps = \lL_\eps \phi_\eps - \phi_{\eps}^{\diamond (2n-1)} + \xi\;
\end{equation}
on a nice bounded domain $\dD \subset \R^2$ with Dirichlet boundary condition. Here, $\xi$ is the two dimensional space-time white noise on $\dD$, $\lL_\eps = \div \big( a(\cdot/\eps) \nabla \big)$ is a divergence-form elliptic operator with periodic and bounded elliptic matrix $a$ and oscillation parameter $\eps$, and $\phi_\eps^{\diamond (2n-1)}$ is the Wick product taken with respect to Gaussian structure induced by the linear part of the equation. Throughout, we make the following assumption on $a$. 

\begin{assumption}\label{as:a}
    The coefficient matrix $a: \R^2 \rightarrow \R^{2 \times 2}$ is $1$-periodic, uniformly elliptic, symmetric and H\"{o}lder continuous. 
\end{assumption}

With these assumptions on $a$, the homogenised operator $\lL_0$ of $\lL_\eps$ is given by
\begin{equation} \label{eq:homogenised_operator}
    \lL_0 = \div (\widehat{a} \nabla)\;,
\end{equation}
where $\widehat{a}$ is a constant $2 \times 2$ positive-definite symmetric matrix given by
\begin{equation} \label{eq:homogenised_coefficient}
    \widehat{a} = \int_{\TT^2} a \big( \id + \nabla \chi, (y) \big) {\rm d}y\;,
\end{equation}
where $\chi = (\chi_1, \chi_2)$ solves the elliptic PDE
\begin{equation*}
    \div \big( a (e_k + \nabla \chi_k) \big) = 0\;, \quad \int_{\T^2} \chi_k {\rm d}y = 0
\end{equation*}
on the torus $\T^2$ for $k=1,2$. 

Although the homogenisation of $\lL_\eps$ to $\lL_0$ is standard, the motivation to study \eqref{eq:main} is to combine two singular limiting procedures, homogenisation and renormalisation, in one single problem. Indeed, consider a regularised version of \eqref{eq:main} given by
\begin{equation} \label{eq:regular_osci}
	\d_t \phi_\eps^{(\delta)} = \lL_\eps \phi_\eps^{(\delta)} - H_{2n-1} \big( \phi_\eps^{(\delta)}; \, C_{\eps}^{(\delta)} \big) + \xi^{(\delta)}\;,
\end{equation}
where $\xi^{(\delta)}$ is a smooth approximation to the space-time white noise $\xi$ at scale $\delta$, and $H_{2n-1}(\cdot \, ; \, C_{\eps}^{(\delta)})$ is the $(2n-1)$-th Hermite polynomial with a proper choice of variance $C_{\eps}^{(\delta)}$ to be specified later. The two natural ways of taking limits are: $\eps \rightarrow 0$ first, and then $\delta \rightarrow 0$, and the other way around. The first way of taking limits is now standard. Let $\psi_{\eps}^{(\delta)}$ be the stationary-in-time solution to the linear equation
\begin{equation*}
    \d_t \psi_\eps^{(\delta)} = \lL_\eps \psi_\eps^{(\delta)} + \xi^{(\delta)}\;.
\end{equation*}
Choose $C_{\eps}^{(\delta)}(x) = \E |\psi_\eps^{(\delta)}(t,x)|^2$ (it is independent of $t$ since $\psi_\eps^{(\delta)}$ is stationary-in-time), then the remainder $u_\eps^{(\delta)} = \phi_\eps^{(\delta)} - \psi_\eps^{(\delta)}$ satisfies the equation
\begin{equation*}
    \d_t u_\eps^{(\delta)} = \lL_\eps u_\eps^{(\delta)} - \sum_{k=0}^{2n-1} \begin{pmatrix} 2 n - 1 \\ k \end{pmatrix} (\psi_\eps^{(\delta)})^{\diamond (2n-1-k)} (u_\eps^{(\delta)})^{k}\;,
\end{equation*}
where $(\psi_\eps^{(\delta)})^{\diamond (2n-1-k)}$ is the $(2n-1-k)$-th Wick product of the Gaussian $\psi_\eps^{(\delta)}$. For fixed $\delta>0$, since $\psi_\eps^{(\delta)}$ is smooth, taking the $\eps \rightarrow 0$ limit are standard periodic homogenisation for both $\psi_\eps^{(\delta)}$ and $u_\eps^{(\delta)}$. Indeed, for every $\delta > 0$, we have
\begin{equation*}
    (\psi_\eps^{(\delta)})^{\diamond (2n-1-k)} \rightarrow (\psi_0^{(\delta)})^{\diamond (2n-1-k)}
\end{equation*}
in $\cC([0,1]; \cC^{1-\kappa}(\dD))$ as $\eps \rightarrow 0$, where $\psi_0^{(\delta)}$ is the stationary solution to the homogenised equation
\begin{equation*}
    \d_t \psi_0^{(\delta)} = \lL_0 \psi_0^{(\delta)} + \xi^{(\delta)}\;,
\end{equation*}
where $\lL_0 = \div (\widehat{a} \nabla)$ is the homogenised operator in \eqref{eq:homogenised_operator} with $\widehat{a}$ being the homogenised constant elliptic matrix given in \eqref{eq:homogenised_coefficient}. Then as soon as the initial data $\phi_\eps(0)$ to \eqref{eq:main} converges to some $\phi_0(0)$ in $\cC^{-\kappa}(\dD)$, it follows that $u_\eps^{(\delta)} \rightarrow u_0^{(\delta)}$, where $u^{(\delta)}_{0}$ satisfies the homogenised equation
\begin{equation*}
    \d_t u_0^{(\delta)} = \lL_0 u_0^{(\delta)} - \sum_{k=0}^{2n-1} \begin{pmatrix} 2 n - 1 \\ k \end{pmatrix} (\psi_0^{(\delta)})^{\diamond (2n-1-k)} (u_0^{(\delta)})^{k}
\end{equation*}
with initial data $u^{(\delta)}_0(0) = \phi_0(0) - \psi_0(0)$. Now, it is standard da Prato-Debussche method (\cite{DaPratoDebussche2003}) to show that $\psi_0^{(\delta)} \rightarrow \psi_0$ in $\cC([0,1]; \cC^{-\kappa}(\dD))$ and $u_0^{(\delta)} \rightarrow u_0$ in $\cC^{2-\kappa}(\dD)$ for fixed time (with a weight at $t=0$), where $\psi_0$ is the stationary solution to
\begin{equation} \label{eq:psi_0}
    \d_t \psi_0 = \lL_0 \psi_0 + \xi\;,
\end{equation}
and $u_0$ solves the equation
\begin{equation} \label{eq:u_0}
    \d_t u_0 = \lL_0 u_0 - \sum_{k=0}^{2n-1} \begin{pmatrix} 2 n - 1 \\ k \end{pmatrix} \psi_0^{\diamond (2n-1-k)} u_0^{k}\;, \quad u(0) = \phi_0(0) - \psi_0(0)\;.
\end{equation}
Let us now consider the other limiting procedure. For fixed $\eps>0$, one can still send $\delta \rightarrow 0$ to obtain the convergence
\begin{equation*}
    \psi_\eps^{(\delta)} \rightarrow \psi_\eps\;, \qquad u_\eps^{(\delta)} \rightarrow u_\eps\;
\end{equation*}
in proper spaces respectively, where $\psi_\eps$ and $u_\eps$ satisfies the same equations for $\psi_\eps^{(\delta)}$ and $u_\eps^{(\delta)}$ with $\xi^{(\delta)}$ replaced by $\xi$ and $(\psi_\eps^{(\delta)})^{\diamond (2n-1-k)}$ replaced by $\psi_\eps^{\diamond (2n-1-k)}$. Note that at this stage, the estimates here are allowed to depend on $\eps$. The natural question is what is the behaviour of $\psi_\eps$ and $u_\eps$ as $\eps \rightarrow 0$, which is a periodic homogenisation problem for a singular stochastic PDE. Our main theorem is the following. 

\begin{thm} \label{th:main}
    Let $\dD \subset \R^2$ be a connected, open, bounded domain with smooth boundary. Let $\psi_\eps$ be the stationary-in-time solution to
    \begin{equation*}
        \d_t \psi_\eps = \lL_\eps \psi_\eps + \xi\;, 
    \end{equation*}
    and $u_\eps$ be the solution to
    \begin{equation*}
    \d_t u_\eps = \lL_\eps u_\eps - \sum_{k=0}^{2n-1} \scriptsize{\begin{pmatrix} 2 n - 1 \\ k \end{pmatrix}} (\psi_\eps)^{\diamond (2n-1-k)} u_\eps^{k}\;, \quad u_\eps(0) = \phi_\eps(0) - \psi_\eps(0)\;.
    \end{equation*}
    Both $\psi_\eps$ and $u_\eps$ satisfy Dirichlet boundary condition on $\d \dD$. Then, $\psi_\eps \rightarrow \psi$ in probability in $\cC([0,T]; \cC^{-\kappa}(\dD))$, where $\psi$ is the stationary solution to \eqref{eq:psi_0} with Dirichlet boundary condition. Let $\beta, \kappa>0$ be sufficiently small (depending on $n$). If $\phi_\eps(0) \rightarrow \phi(0)$ in $\cC^{-\beta}(\dD)$, then for every $p, T > 0$, we have
    \begin{equation*}
        \E \Big( \sup_{t \in [0,T]} t^{\frac{1+\beta}{2}} \|u_\eps(t) - u_0(t)\|_{\cC^{1-\kappa}(\dD)} \Big)^p \rightarrow 0\;,
    \end{equation*}
    where $u_0$ solve the equations \eqref{eq:u_0} with initial data $u_0(0) = \phi_0(0) - \psi_0(0)$ and Dirichlet boundary condition. The process $\phi_0 = \psi_0 + u_0$ is the solution the standard dynamical $\phi^{2n}_2$ model
    \begin{equation} \label{eq:phi_limit}
        \d_t \phi_0 = \lL_0 \phi_0 - \phi_0^{\diamond (2n-1)} + \xi
    \end{equation}
    with initial data $\phi_0(0)$ and Dirichlet boundary condition. 
\end{thm}

Near the completion of this article, we learned the forthcoming work by Hairer and Singh (\cite{st_homo_phi42}), who studied a very similar question as in this article, but in a more general formulation allowing space-time periodic coefficients and joint limits $(\eps,\delta) \rightarrow (0,0)$ in arbitrary ways. 

Although these seem to be the first works combining homogenisation and renormalisations in singular nonlinear stochastic PDEs, there are homogenisation problems with singular right hand side (below deterministic regularity threshold) considered before. Hairer and Kelley (\cite{SPDE_multiscale}) studied a periodic homogenisation problem with general elliptic operator and space-time white noise in dimension $1$. The interesting phenomena is that the solution converges to a linear stochastic heat equation with the expected homogenised coefficient, but with a different constant multiplying the noise in the limit. They further showed that the multiplicative constant in front of the noise is $1$ if and only if the oscillatory operator is of divergence form, which is the situation considered in this article. In \cite{add_functional_rp},  Deuschel, Orenshtein and Perkowski studied an interesting situation in which the second order iterated integral in rough paths and homogenisation do not commute. In \cite{Stoch_homo_Gaussian}, Chiarini and Ruszel studied stochastic homogenisation problem with independent singular Gaussian random fields on the right hand side, and obtained convergence to the (expected) homogenised stochastic heat equation, but with a slower rate than the standard situation. More recently, Kremp and Perkowski (\cite{homo_sde_levy}) studied a periodic homogenisation problem for singular L\'evy SDEs beyond Young regime, and proved central limit theorems for both the diffusion case and $\alpha$-stable case. These works lead to natural questions on how the statement in the current context can be extended if the oscillatory operator is not of divergence form, or if the coefficient field $a$ is random, or if the homogenisation and renormalisation procedure do not commute, etc.

\begin{rmk}
    The stochastic PDE \eqref{eq:main} and its regularised version \eqref{eq:regular_osci} are not translation invariant since $a$ depends on $x$. The choice of renormalisations $\{C_\eps^{(\delta)}(\cdot)\}_{\delta>0}$ then also depends on the spatial location $x$. This in principle gives an infinite degree of freedom in the renormalisation procedure, in contrast to the finite dimensionality in translation invariant situations. 
    
    In the recent article \cite{SPDE_variable_nontrans}, Singh gave a natural choice of a finite dimensional family of solutions in a large class of non-translation invariant situations with explicit counter-terms, including space-time dependent coefficient. The way to achieve this is to freeze the coefficient locally in space-time, and the problem becomes translation invariant with the frozen coefficient. This works well when the coefficient $a$ has sufficient regularity, which is not the case for homogenisation if one wants uniformity in the oscillatory parameter in the coefficient. 
    
    Let us consider the specific example \eqref{eq:regular_osci} with $n=2$ and on the torus $\dD = \T^2$ (so $\eps = \eps_N = \frac{1}{N}$ to keep the problem on the torus). We have
    \begin{equation*}
        \d_t \phi_\eps^{(\delta)} = \lL_\eps \phi_\eps^{(\delta)} - (\phi_\eps^{(\delta)})^3 + 3 C_{\eps}^{(\delta)} \phi_\eps^{(\delta)} + \xi^{(\delta)}\;, 
    \end{equation*}
    where $\xi^{(\delta)}$ is the specific regularisation of $\xi$ by heat kernel (see \cite{SPDE_variable_nontrans}) such that
    \begin{equation*}
        \xi^{(\delta)}(t,x) = \int_{\R} \rho^{(\delta)}(t-s) \int_{\dD} G_\eps(t-s+\delta^2, x, y) \xi(s,y) {\rm d}y {\rm d}s\;, 
    \end{equation*}
    where $\rho$ is an even positive $\cC_{c}^{\infty}(\R)$ function integrating to $1$, and $\rho^{(\delta)}(s) = \delta^{-2} \rho(s/\delta^2)$. Here, the choice of regularisation also makes $\xi^{(\delta)}$ depends on $\eps$, but we omit it in notation for simplicity. 

    The counter-term $\widetilde{C}_{\eps}^{(\delta)}(x)$ given in \cite{SPDE_variable_nontrans} to this equation, as a function in $x \in \T^2$, is proportional to $\det(a(x/\eps))^{-\frac{1}{2}}$. The counter-term $C_{\eps}^{(\delta)}(x)$ considered in this article is the variance of the solution to the stationary linear equation. One implication of Singh's result is that for every fixed $\eps>0$, $C_{\eps}^{(\delta)} - \widetilde{C}_{\eps}^{(\delta)}$ converges uniformly on $\T^2$ to a continuous function as $\delta \rightarrow 0$. 
    
    On the other hand, by explicit computation, one can show there exists a function $\lambda(\delta) \sim |\log \delta|$ such that for every $\delta>0$, one has
    \begin{equation*}
        \int_{\T^2} C_{\eps}^{(\delta)}(x) {\rm d}x \rightarrow \frac{\lambda(\delta)}{\sqrt{\det{\widehat{a}}}} \;, \qquad \int_{\T^2} \widetilde{C}_{\eps}^{(\delta)}(x) {\rm d}x \rightarrow \int_{\T^2} \frac{\lambda(\delta)}{\sqrt{\det \big(a(y) \big)}} {\rm d}y
    \end{equation*}
    as $\eps \rightarrow 0$. The two coefficients multiplying $\lambda(\delta)$ are not equal for generic non-constant $a$, hence the two renormalisations diverge far apart (even in the space of distributions) if one wants uniformity in $\eps$. In \cite{st_homo_phi42}, another renormalisation constant (in addition to the renormalisation function $\widetilde{C}_{\eps}^{(\delta)}$) was added to ensure convergence to the limit. 
\end{rmk}

\subsection*{Organisation of the article}

The rest of the article is organised as follows. In Section~\ref{sec:linear}, we prove convergence of the Wick power $\psi_\eps^{\diamond m}$ to their corresponding homogenised ones as $\eps \rightarrow 0$. In Section~\ref{sec:remainder}, we analyse the remainder equation $u_\eps$, giving its existence and apriori bounds. In Section~\ref{sec:convergence}, we show that $u_\eps$ homogenises to the expected limit $u_0$, and hence also $\phi_\eps$ to its expected limit $\phi_0$. In the appendices, we give some estimates on the operator $\lL_\eps$ and $\lL_0$ and their Green functions, as well as some useful functional inequalities.

\subsection*{Notations}

Let $G_\eps (t;x,y)$ be the Green's function of the parabolic operator $\d_t - \lL_\eps$ in the sense that
\begin{equation*}
    (e^{t \lL_\eps} f)(x) = \int_{\dD} G_\eps(t;x,y) f(y) {\rm d}y\;.
\end{equation*}
Throughout, we fix $\kappa > 0$ to be a sufficiently small number ($\kappa < \frac{1}{100 n}$ would suffice). For $\alpha >0$, let
\begin{equation*}
    \bB^\alpha:= \big\{ \varphi \in \cC_c^{\alpha}(\dD):  \|\varphi\|_{\cC^\alpha(\dD)} \leq 1 \big\}\;.
\end{equation*}
and
\begin{equation*}
    \begin{split}
        \varphi^{\lambda}_{x}(y) \triangleq \frac{1}{\lambda^2} \varphi \Big(\frac{y-x}{\lambda} \Big)
    \end{split}
\end{equation*}
For $\alpha<0$, the $\cC^\alpha$-norm is defined by
\begin{equation*}
    \|f\|_{\cC^\alpha(\dD)} := \sup_{x \in \dD} \sup_{\lambda} \sup_{\|\varphi\|_{\bB^\alpha} \leq 1} \lambda^{-\alpha} |\scal{f, \varphi_x^\lambda}|\;,
\end{equation*}
where the supremum in $\lambda$ is taken over the range $\lambda \in (0,1)$ with the additional constraint $\text{supp}(\varphi_x^{\lambda}) \subset \dD$.

\subsection*{Acknowledgement}

W. Xu was supported by National Science Foundation China via the standard project grant (no. 8200906145) and the Ministry of Science and Technology via the National Key R\&D Program of China (no. 2020YFA0712900).

\section{The stationary linear part}
\label{sec:linear}

Recall that the stationary linear solution $\psi_\eps$ has the form
\begin{equation} \label{eq:linear_stationary}
	\psi_{\eps}(t) :=  \int_{-\infty}^{t} e^{(t-r) \lL_\eps} \xi(r) {\rm d}r\;
\end{equation}
and the stationary solution $\psi_0$ to the homogenised linear equation is given by
\begin{equation} \label{eq:linear_stationary_homogenised}
	\psi_0(t) := \int_{-\infty}^{t} e^{(t-r) \lL_0} \xi(r) {\rm d}r\;.
\end{equation}
The main result of this section is the following theorem regarding the convergence of Wick powers of the stationary linear solutions. 

\begin{thm} \label{thm:linear_stationary}
For the stationary linear solution $\psi_\eps$ given in \eqref{eq:linear_stationary}, every integer $m$, and every $\delta \in (0,\frac{1}{4})$, we have
\begin{equation} \label{eq:linear_wick_bound}
    \E \big| \bigscal{ \big(\psi_{\eps}^{\diamond m}(t) - \psi_{\eps}^{\diamond m}(s) \big) - \big(\psi_0^{\diamond m}(t) - \psi_{0}^{\diamond m}(s) \big), \varphi_x^\lambda} \big|^{2} \lesssim \big( |t-s| \wedge 1 \big)^{\delta} \eps^{\delta} \lambda^{-8 \delta}\;,
\end{equation}
uniformly over $\eps \in [0,1]$, $\lambda \in [0,1]$, $s,t \in \R^{+}$, $x \in \dD$ and $\varphi \in L^{\infty}(\dD)$ such that $\|\varphi\|_{L^\infty (\dD)} \leq 1$ and $supp(\varphi) \subset \dD$. As a consequence, for every integer $m$ and every $p \geq 1$, we have
    \begin{equation} \label{eq:linear_wick_convergence}
    \begin{split}
                \E \sup_{t \in [0,T]} \|\psi_\eps^{\diamond m}(t,\cdot) - \psi_0^{\diamond m}(t,\cdot)\|_{\cC^{-\kappa}(\dD)}^{p} \rightarrow 0\\
        \sup_{t \in \R} \E  \|\psi_\eps^{\diamond m}(t,\cdot) - \psi_0^{\diamond m}(t,\cdot)\|^p_{\cC^{-\kappa}(\dD)} \rightarrow 0
    \end{split}
    \end{equation}
    as $\eps \rightarrow 0$. 
\end{thm}

It suffices to prove \eqref{eq:linear_wick_bound}. The bounds in \eqref{eq:linear_wick_convergence} follows from \eqref{eq:linear_wick_bound} and Kolmogorov's criterion. The rest of the section is devoted to the proof of \eqref{eq:linear_wick_bound}. 

\begin{lem} \label{lem:corr_linear}
    For every $\delta>0$, we have
    \begin{equation*}
        \E \big( \psi_\eps(s,x) \psi_\eps(t,y) \big) \lesssim_\delta \big( \sqrt{|t-s|} + |x-y| \big)^{-\delta}
    \end{equation*}
    uniformly over all $x, y \in \dD$ and all $s, t \geq 0$. 
\end{lem}
\begin{proof}
We have the expression
\begin{equation*}
    \E \big( \psi_\eps(s,x) \psi_\eps(t,y) \big) = \int_{-\infty}^{s \wedge t} \Big(\int_{\dD} G_{\eps}(s-r,x,z) G_{\eps}(t-r,y,z) {\rm d}z \Big) {\rm d}r\;.
\end{equation*}
The claim then follows directly from the pointwise bounds on the Green's function $G_\eps$ in \eqref{eq:green_pointwise}. 
\end{proof}

\begin{prop} \label{prop:linear_wick_fixed_time}
    For every integer $m$ and every $\delta \in (0,1)$, we have
    \begin{equation*}
        \E |\scal{\psi_\eps^{\diamond m}(t), \varphi^\lambda_x}|^{2} \lesssim_{\delta,m} \lambda^{-2\delta}
    \end{equation*}
    uniformly over $\eps \in [0,1]$, $\lambda \in (0,1)$, $t \in \R^+$, $x \in \dD$ and $\varphi \in \cC^{\kappa}(\dD)$. 
\end{prop}
\begin{proof}
    The assertion is a direct consequence of the correlation bound in Lemma~\ref{lem:corr_linear} and Wick's formula, and replacing $m \delta$ by $2 \delta$. 
\end{proof}

\begin{prop} \label{prop:linear_wick_time_difference}
    For every integer $m$ and every $\delta \in (0,1)$, we have
    \begin{equation*}
        \E |\scal{\psi_\eps^{\diamond m}(t) - \psi_\eps^{\diamond m}(s), \varphi_x^\lambda}|^2 \lesssim_{m,\delta} (|t-s| \wedge 1)^{\delta} \lambda^{-3\delta}
    \end{equation*}
    uniformly over $\eps \in [0,1]$, $\lambda \in (0,1)$, $s,t \in \R^+$ and $x \in \dD$. 
\end{prop}
\begin{proof}
    If $|t-s|>1$, the assertion follows directly from Proposition~\ref{prop:linear_wick_fixed_time}. So we only need to consider the case $|t-s|<1$. We assume without loss of generality that $0<s<t$. It suffices to show the correlation bound
    \begin{equation*}
        \E \big( \psi_{\eps}^{\diamond m}(t,y) - \psi_{\eps}^{\diamond m}(s,y) \big) \big( \psi_{\eps}^{\diamond m}(t,z) - \psi_{\eps}^{\diamond m}(s,z) \big) \lesssim (t-s)^{\delta} |y-z|^{-3\delta}\;.
    \end{equation*}
    By Wick's formula, the bound in Lemma~\ref{lem:corr_linear}, and replacing $(m-1) \delta$ by $\delta$, it suffices to show the bounds
    \begin{equation*}
        \begin{split}
        \E \big[ \psi_{\eps}(t,y) \big( \psi_{\eps}(t,z) - \psi_{\eps}(s,z) \big) \big] &\lesssim (t-s)^{\delta} |y-z|^{-2\delta}\;,\\
        \E \big[ \psi_{\eps}(s,y) \big( \psi_{\eps}(t,z) - \psi_{\eps}(s,z) \big)  \big] &\lesssim (t-s)^{\delta} |y-z|^{-2\delta}\;.
        \end{split}
    \end{equation*}
    We give details for the first one. We have the expression
    \begin{equation*}
        \begin{split}
        \E \big[ \psi_{\eps}(t,y) &\big( \psi_{\eps}(t,z) - \psi_{\eps}(s,z) \big) \big] = \int_{s}^{t} \int_{\dD} G_{\eps}(t-r,y,w) G_{\eps}(t-r,z,w) {\rm d}w {\rm d}r\\
        &+ \int_{-\infty}^{s} \int_{\dD} G_{\eps}(t-r,y,w) \big( G_{\eps}(t-r,z,w) - G_{\eps}(s-r,z,w) \big) {\rm d}w {\rm d}r\;.
        \end{split}
    \end{equation*}
    The desired bound then follows from the pointwise bound \eqref{eq:green_pointwise} and difference bound \eqref{eq:green_gradient_pointwise} for the Green function. The other correlation bound follows in essentially the same way. This completes the proof of the proposition. 
\end{proof}

Before we state the bound for $\psi_\eps^{\diamond m}(t) - \psi_0^{\diamond m}(t)$ for general $m \geq 1$, we first give a few preliminary lemmas. For $\eps_1, \eps_2 \in [0,1]$, let
\begin{equation*}
    \rho_{\eps_1, \eps_2}(x,y) := \E \big( \psi_{\eps_1}(t,x) \psi_{\eps_2}(t,y) \big) = \int_{\R^+} \int_{\dD} G_{\eps_1}(r,x,z) G_{\eps_2}(r,y,z) {\rm d}z {\rm d}r\;,
\end{equation*}
which is independent of $t$. We have the following lemma. 

\begin{lem} \label{lem:rho_uniform_bound}
    For every $p \geq 1$, there exists $C>0$ depending on $p$ and the domain $\dD$ only such that
    \begin{equation*}
        \|\rho_{\eps_1, \eps_2}\|_{L^p(\dD \times \dD)} \leq C\;.
    \end{equation*}
    In particular, the bound is uniform in $\eps_1, \eps_2 \in [0,1]$. 
\end{lem}
\begin{proof}
    It suffices to decompose the range of integration of $r \in \R^+$ into $r \in [0,1]$ and $r \in [1,+\infty)$ and apply the pointwise bound \eqref{eq:green_pointwise}. 
\end{proof}

\begin{lem} \label{lem:rho_difference_bound}
    For every $\delta>0$, we have
    \begin{equation*}
        \|\rho_{\eps_1,\eps_2} - \rho_{\tilde{\eps}_1,\tilde{\eps}_2}\|_{L^{2}(\dD^2)} \lesssim_\delta \big( \max\{\eps_1, \eps_2, \tilde{\eps}_1, \tilde{\eps}_2\} \big)^{1-\delta}\;,
    \end{equation*}
    uniformly in the four parameters taken in $[0,1]$. As a consequence, for every $\delta > 0$ and $p \geq 2$, we have
    \begin{equation*}
        \|\rho_{\eps_1,\eps_2} - \rho_{\tilde{\eps}_1,\tilde{\eps}_2}\|_{L^{p}(\dD^2)} \lesssim_{p,\delta} \big( \max\{\eps_1, \eps_2, \tilde{\eps}_1, \tilde{\eps}_2\} \big)^{1-\delta}\;.
    \end{equation*}
\end{lem}
\begin{proof}
    The second bound follows from the first bound and Lemma~\ref{lem:rho_uniform_bound}, so we only need to consider the first bound. Since
    \begin{equation*}
        \rho_{\eps_1, \eps_2} - \rho_{\tilde{\eps}_1, \tilde{\eps}_2} = (\rho_{\eps_1, \eps_2} - \rho_{\eps_1,0}) + (\rho_{\eps_1, 0} - \rho_{\eps_1, \tilde{\eps}_2}) + (\rho_{\eps_1, \tilde{\eps}_2} - \rho_{0, \tilde{\eps}_2}) + (\rho_{0, \tilde{\eps}_2} - \rho_{\tilde{\eps}_1, \tilde{\eps}_2})\;,
    \end{equation*}
    it suffices to consider the situation $\eps_1 = \eps$, $\tilde{\eps}_1 = 0$ and $\eps_2 = \tilde{\eps}_2 = \eta \in [0,1]$. We have the expression
    \begin{equation*}
        \rho_{\eps, \eta}(x,y) - \rho_{0,\eta}(x,y) = \int_{\R^+} F_{\eps, \eta, r}(x,y) {\rm d}r\;,
    \end{equation*}
    where
    \begin{equation} \label{eq:operator_double_variable}
        \begin{split}
        F_{\eps,\eta,r}(x,y) &= \int_{\dD} \big( G_\eps(r,x,z) - G_0(r,x,z) \big) G_{\eta}(r,y,z) {\rm d}z\\
        &= \big( (e^{r \lL_\eps} - e^{r \lL_0}) G_{\eta}(r,y,\cdot) \big)(x)\;.
        \end{split}
    \end{equation}
    By triangle inequality, we have
    \begin{equation*}
        \|\rho_{\eps,\eta} - \rho_{0,\eta}\|_{L^2(\dD^2)} \leq \int_{\R^+} \|F_{\eps,\eta,r}\|_{L^2(\dD^2)} {\rm d}r\;.
    \end{equation*}
    By the expression \eqref{eq:operator_double_variable} and the semi-group bound \eqref{eq:semi_group_difference}, we have
    \begin{equation*}
        \int_{\dD} |F_{\eps,\eta,r}(x,y)|^2 {\rm d}x \lesssim \Big(\frac{\eps e^{-cr}}{\sqrt{r} + \eps} \Big)^2 \cdot \|G_{\eta}(r,y,\cdot)\|_{L^2(\dD)}^2\;.
    \end{equation*}
    This in turn gives
    \begin{equation*}
        \|\rho_{\eps,\eta} - \rho_{0,\eta}\|_{L^2(\dD^2)} \lesssim \int_{\R^+} \frac{\eps e^{-cr}}{\sqrt{r} + \eps}  \cdot \|G_{\eta}(r)\|_{L^2(\dD^2)} {\rm d}r \lesssim \int_{\R^+} \frac{\eps e^{-cr}}{\sqrt{r} (\sqrt{r} + \eps)}  {\rm d}r\;,
    \end{equation*}
    where the last inequality follows from the Green's function estimate~\eqref{eq:green_pointwise} and hence the bound
    \begin{equation*}
        \sup_{\eta \in [0,1]} \|G_{\eta}(r,\cdot, \cdot)\|_{L^2(\dD^2)} \lesssim \frac{1}{\sqrt{r}}
    \end{equation*}
    for $r \in \R^+$. This implies
    \begin{equation*}
        \|\rho_{\eps,\eta} - \rho_{0,\eta}\|_{L^2(\dD^2)} \lesssim \eps |\log \eps|
    \end{equation*}
    uniformly in $\eta \in [0,1]$. The conclusion then follows. 
    \end{proof}

\begin{prop} \label{prop:linear_wick_difference}
    For every $m$ and $\kappa>0$, we have
    \begin{equation*}
        \E |\scal{\psi_{\eps}^{\diamond m}(t) - \psi_{0}^{\diamond m}(t), \varphi}|^2 \lesssim_{m,\kappa} \eps^{1-\kappa} \|\varphi\|_{L^2(\dD)}^2\;,
    \end{equation*}
    uniformly in $\eps \in [0,1]$ and $\varphi \in L^2(\dD)$. 
\end{prop}
\begin{proof}
    We have the identity
    \begin{equation*}
        \E \big| \scal{\psi_\eps^{\diamond m}(t) - \psi_0^{\diamond m}(t), \varphi} \big|^2 = m! \iint\limits_{\dD^2} \varphi(x) \varphi(y) \big( \rho_{\eps,\eps}^m - \rho_{\eps,0}^{m} - \rho_{0,\eps}^{m} + \rho_{0,0}^{m} \big) {\rm d}x {\rm d}y\;.
    \end{equation*}
    The claim then follows directly from Lemmas~\ref{lem:rho_uniform_bound} and~\ref{lem:rho_difference_bound}. 
\end{proof}

\begin{rmk}
    With some extra effort, one can possibly improve the error rate arbitrarily close to $\eps^2$. But since it does not affect the main statement, we leave the proposition as it is and choose not to pursue further.  
\end{rmk}

\begin{proof} [Proof of Theorem~\ref{thm:linear_stationary}]
    We are now ready to prove Theorem~\ref{thm:linear_stationary}. By Propositions~\ref{prop:linear_wick_fixed_time}  and~\ref{prop:linear_wick_difference}, the left hand side of \eqref{eq:linear_wick_bound} is bounded by
    \begin{equation*}
        \lambda^{-2\delta} \, \wedge \, \eps^{1-\delta} \lambda^{-2} \lesssim \eps^{\delta} \lambda^{-4\delta}\;.
    \end{equation*}
    Combining with Proposition~\ref{prop:linear_wick_time_difference} and replacing $\delta$ by $2\delta$, we can further bound it by
    \begin{equation*}
        \eps^{2 \delta} \lambda^{-8\delta} \, \wedge \, (|t-s| \wedge 1)^{2 \delta} \lambda^{-6\delta} \lesssim \eps^{\delta} (|t-s| \wedge 1)^{\delta} \lambda^{-8\delta}\;,
    \end{equation*}
    which is precisely \eqref{eq:linear_wick_bound}. Since the left hand side of \eqref{eq:linear_wick_bound} is in Wiener chaos of order $m$, the same bound holds if we replace the second moment by $p$-th moment for arbitrary $p$ (note that $\delta$ can be arbitrary). The two convergence statements in Theorem~\ref{thm:linear_stationary} follow from the $p$-th moment version of \eqref{eq:linear_wick_bound} and Kolmogorov's continuity criterion. 
\end{proof}

\section{The remainder equation}
\label{sec:remainder}

Let
\begin{equation*}
    u_\eps = \phi_\eps - \psi_\eps\;, \qquad Y_\eps(t) := u_\eps(t) - e^{t \lL_\eps} u_\eps(0)\;.
\end{equation*}
The main theorem of this section is the following. 

\begin{thm} \label{thm:global_existence}
    For every $T>0$, every $p \geq 1$ and every $\theta \in (0,1)$, there exists $C = C(T,p,\theta)$ such that
    \begin{equation} \label{eq:bound_Y}
        \E \sup_{t \in [0,T]} \|Y_\eps(t)\|_{L^p(\dD)}^{p} \leq C\;,
    \end{equation}
    and
    \begin{equation} \label{eq:bound_u}
        \E \sup_{t \in [0,T]} \big[ (t \wedge 1)^{p\theta} \, \|u_\eps(t)\|_{L^p(\dD)}^{p} \big] < C\;.
    \end{equation}
    Both bounds are uniform in $\eps$. Furthermore, we have
    \begin{equation} \label{eq:bounds_Y_u_global}
    \begin{split}
        \sup_{\eps \in [0,1]} \sup_{t \geq 0} \E \|Y_\eps(t)\|_{L^p(\dD)} < +\infty\;,\\
        \sup_{\eps \in [0,1]} \sup_{t \geq 0} \big[ (t \wedge 1)^{\theta} \, \E \|u_\eps(t)\|_{L^p(\dD)} \big] < +\infty\;.
    \end{split}
    \end{equation}
\end{thm}

\subsection{Local existence}

Let $T>0$ be fixed, and $\{g_\eps^{(k)}\}$ be a class of space-time distributions such that
\begin{equation*}
    \mM_\eps := \max_{0 \leq k \leq 2n-1} \sup_{t \in [0,T]} \big( t^{\frac{1}{5}} \|g_\eps^{(k)}(t)\|_{\cC^{-\kappa}} \big) +\infty\;.
\end{equation*}
In our situation, we use
\begin{equation*}
    g_{\eps}^{(k)}(t) = \big( \psi_\eps(t) + e^{t \lL_\eps} u_\eps(0) \big)^{\diamond (2n-1-k)}\;,
\end{equation*}
although in what follows it can mostly be regarded as general space-time distributions satisfying the bound above. Let $\xX_\tau$ be the norm defined by
\begin{equation*}
    \|Y\|_{\xX_\tau} := \sup_{t \in [0,\tau]} \|Y(t)\|_{W_0^{1,+\infty}(\dD)}\;.
\end{equation*}
We have the following theorem regarding local existence of $Y_\eps$. 

\begin{thm} \label{thm:local_existence}
    For every $\eps>0$, there exists $\tau_\eps > 0$ depending on $\mM_\eps$ only such that there is a unique $Y_\eps \in \xX_{\tau_\eps}$ such that
    \begin{equation} \label{eq:remainder_Duhamel}
        Y_\eps(t) = - \sum_{k=0}^{2n-1} \begin{pmatrix} 2 n - 1 \\ k \end{pmatrix} \int_{0}^{t} e^{(t-r)\lL_\eps} \big( g_{\eps}^{(k)}(r) Y_{\eps}^{k}(r) \big) {\rm d}r\;.
    \end{equation}
    Furthermore, if $T_\eps > 0$ and $Y_{\eps}^{(1)}$ and $Y_{\eps}^{(2)}$ are two solutions satisfying \eqref{eq:remainder_Duhamel} on $[0,T_\eps)$, then $Y_{\eps}^{(1)} = Y_{\eps}^{(2)}$ on $[0,T_\eps)$. 
\end{thm}

\begin{rmk}
    Note that the local existence time $\tau_\eps$ depends on $\mM_\eps$ only. In particular, it will be independent of $\eps$ if $\mM_\eps$ is. Also, the second statement is about uniqueness up to arbitrary fixed time $T_\eps$ (possibly beyond the local existence time $\tau_\eps$). 
\end{rmk}

\begin{proof} [Proof of Theorem~\ref{thm:local_existence}]
    Define the operator $\Gamma_\eps$ by
    \begin{equation*}
        (\Gamma_\eps Y)(t) := - \sum_{k=0}^{2n-1} \begin{pmatrix} 2 n - 1 \\ k \end{pmatrix} \int_{0}^{t} e^{(t-r)\lL_\eps} \big( g_{\eps}^{(k)}(r) Y^{k}(r) \big) {\rm d}r\;.
    \end{equation*}
    We first show that, for sufficiently small $\tau$ depending on $\mM_\eps$ only, $\Gamma_\eps$ is a map from a ball of radius $1$ in $\xX_\tau$ into itself. Indeed, by \ref{eq:semi_group_derivative}, we have
    \begin{equation*}
        \|\nabla (\Gamma_\eps Y)(t)\|_{L^\infty} \lesssim \int_{0}^{t} (t-r)^{-\frac{3}{5}} \|g_\eps^{(k)}(r)\|_{\cC^{-\kappa}} \|Y(r)\|_{\cC^{2\kappa}}^{k} {\rm d}r\;.
    \end{equation*}
    Since $\Gamma_{\eps} Y=0$ on $\d \dD$, we get
    \begin{equation*}
        \|\Gamma_\eps Y\|_{\xX_\tau} \lesssim \big(1 + \mM_\eps\big)^{2n-1} \sup_{t \in [0,\tau]} \int_{0}^{t} (t-r)^{-\frac{3}{5}} r^{-\frac{1}{5}} {\rm d}r \lesssim \tau^{\frac{1}{5}} (1 + \mM_\eps)^{2n-1}\;.
    \end{equation*}
    Taking $\tau \lesssim (1 + \mM_\eps)^{-5(2n-1)}$ (with a sufficiently small proportionality constant), we see that $\Gamma_\eps$ maps the unit ball of $\xX_\tau$ into itself. 

    To show $\Gamma_\eps$ is also a contraction, we first have
    \begin{equation*}
        \begin{split}
        \|(\Gamma_\eps Y_1)(t) - (\Gamma_\eps Y_2)(t)\|_{W^{1,\infty}} \lesssim \sum_{k=1}^{2n-1} \int_{0}^{t} &(t-r)^{-\frac{3}{5}} \|g_{\eps}^{(k)}(r)\|_{\cC^{-\kappa}} \|Y_1(r) - Y_{2}(r)\|_{\cC^{2 \kappa}}\\
        &\big( \|Y_1(r)\|_{\cC^{2\kappa}}^{k-1} + \|Y_2(r)\|_{\cC^{2\kappa}}^{k-1} \big) {\rm d}r\;.
        \end{split}
    \end{equation*}
    Again, since we are restricting $\Gamma_{\varepsilon}$ to the unit ball of $\xX_{\tau}$, we have
    \begin{equation*}
        \|\Gamma_\eps (Y_1) - \Gamma_\eps(Y_2)\|_{\xX_{\tau}} \lesssim \tau^{\frac{1}{5}} (1 + \mM_\eps)^{2n-1} \|Y_1 - Y_2\|_{\xX_{\tau}}. 
    \end{equation*}
    Taking $\tau \lesssim (1 + \mM_\eps)^{-5(2n-1)}$ makes $\Gamma_\eps$ a contraction. This gives the unique local existence up to time $\tau_\eps$. 

    Now, let $T_\eps >0$ and suppose $Y_\eps^{(1)}$ and $Y_\eps^{(2)}$ are two solutions to \eqref{eq:remainder_Duhamel} in $\xX_{T_\eps}$. If they are not identical, then let
    \begin{equation*}
        \sigma_\eps := \sup \big\{ r \in [0,T_\eps): Y_{\eps}^{(1)}(s) = Y_{\eps}^{(2)}(s) \; \text{for all} \; s \in [0,r] \big\}
    \end{equation*}
    be the maximal time up to which two solutions agree. By assumption, we have $\sigma_\eps < T_\eps$. Then, the argument as above shows that
    \begin{equation*}
        \|Y_\eps^{(1)}(t) - Y_\eps^{(2)}(t)\|_{W_{0}^{1,+\infty}} \lesssim (t-\sigma_\eps)^{\frac{1}{5}} (1 +\mM_\eps)^{2n-1} \sup_{r \in [\sigma_\eps, t]} \|Y_{\eps}^{(1)}(r) - Y_{\eps}^{(2)}(r)\|_{W_{0}^{1,+\infty}}\;.
    \end{equation*}
    This shows that there exists $\delta>0$ such that $Y_{\eps}^{(1)} = Y_{\eps}^{(2)}$ on $[\tau_\eps, \tau_\eps + \delta]$, contradicting the definition of $\tau_\eps$. This completes the proof of the theorem. 
\end{proof}

\subsection{A priori bounds}

\begin{lem} \label{lem:differential_inequality}
    Suppose $\{v_\eps\}$ is a family of space-time processes satisfying (pathwise) the equation
    \begin{equation*}
        \d_t v_\eps = \lL_\eps v_\eps - (\psi_\eps + e^{t \lL_\eps} f_\eps + v_\eps)^{\diamond (2n-1)}
    \end{equation*}
    on $[0,T]$, where $\{f_\eps\}$ is a family of spatial processes on $\dD$. Then for every $1 \leq k \leq 2n-1$, there exists $\gamma(k) \geq 1$ such that for every $m \geq 1$, we have
    \begin{equation*}
        \begin{split}
        &\phantom{11}\frac{{\rm d}}{{\rm d}t} \big(\|v_\eps(t)\|_{L^{2m}}^{2m} \big) + \|\nabla (v_\eps^m(t))\|_{L^2}^2 + \|v_\eps(t)\|_{L^{2n+2m-2}}^{2n+2m-2}\\
        &\lesssim \sum_{k=1}^{2n-1} \|(\psi_\eps(t) + e^{t \lL_\eps} f_\eps)^{\diamond k}\|_{\cC^{-\kappa}}^{\gamma(k)}\;,
        \end{split}
    \end{equation*}
    where the proportionality constant is independent of $\eps \in (0,1)$ and $t \geq 0$. 
\end{lem}

For application, we will consider two cases: $v_\eps = u_\eps$ with $f_\eps = 0$, and $v_\eps = Y_\eps$ with $f_\eps = u_\eps(0)$. 

\begin{proof}
    Writing
    \begin{equation*}
        g_{\eps}^{(k)}(t) = \big( \psi_{\eps}(t) + e^{t \lL_\eps} f_\eps \big)^{\diamond (2n-1-k)}
    \end{equation*}
    for simplicity, we have
    \begin{equation*}
        \d_t v_\eps - \lL_\eps v_\eps + v_{\eps}^{2n-1} = - \sum_{k=0}^{2n-2} \begin{pmatrix} 2n-1 \\ k \end{pmatrix} g_{\eps}^{(k)} v_{\eps}^{k}\;.
    \end{equation*}
    Multiplying both sides by $v_{\eps}^{2m-1}$, integrating over $x \in \dD$, and using ellipticity of $\lL_\eps$, we get
    \begin{equation} \label{eq:differential_inequality_first}
        \begin{split}
        &\phantom{111}\frac{{\rm d}}{{\rm d}t} \big(\|v_\eps(t)\|_{L^{2m}}^{2m} \big) + \|\nabla (v_\eps^m(t))\|_{L^2}^2 + \|v_\eps(t)\|_{L^{2n+2m-2}}^{L^{2n+2m-2}}\\
        &\lesssim \sum_{k=0}^{2n-2} \|g_{\eps}^{(k)}(t)\|_{\cC^{-\kappa}} \|\nabla \big( v_{\eps}^{2m-1+k}(t) \big)\|_{L^1}^{\kappa} \|v_{\eps}^{2m-1+k}(t)\|_{L^1}^{1-\kappa}\;,
        \end{split}
    \end{equation}
    where we have used Proposition~\ref{prop:B11_bound} and that $v_{\eps}|_{\d \dD} = 0$ in the last inequality. Note that it is important here that the sum over $k$ on the right hand side is up to $2n-2$ rather than $2n-1$. 

    For the term involving the gradient, we have
    \begin{equation*}
        \|\nabla (v_{\eps}^{2m-1+k})\|_{L^1} \lesssim \|\nabla (v_\eps^m) \cdot v_{\eps}^{m-1+k}\|_{L^1} \lesssim \|\nabla (v_{\eps}^{m})\|_{L^2} \|v_{\eps}^{m-1+k}\|_{L^2}\;.
    \end{equation*}
    For the last term on the right hand side, we have
    \begin{equation*}
        \|v_{\eps}^{m-1+k}\|_{L^2} = \|v_{\eps}^{m}\|_{L^{\frac{2(m-1+k)}{m}}}^{1+\frac{k-1}{m}} \lesssim \|\nabla (v_{\eps}^m)\|_{L^2}^{1 + \frac{k-1}{m}}\;,
    \end{equation*}
    where in the last inequality we used the embedding of $H^1$ into $L^p$ for any $p < +\infty$ in dimension two. Hence, we can bound each term in the sum on the right hand side of \eqref{eq:differential_inequality_first} by
    \begin{equation*}
        \begin{split}
        &\phantom{111}\|g_{\eps}^{(k)}\|_{\cC^{-\kappa}} \|\nabla (v_{\eps}^m)\|_{L^2}^{(2+\frac{k-1}{m})\kappa} \|v_{\eps}\|_{L^{2m-1+k}}^{(2m-1+k)(1-\kappa)}\\
        &\leq C_{\mu} \|g_{\eps}^{(k)}\|_{\cC^{-\kappa}}^{\gamma(k)} + \mu \big( \|\nabla (v_{\eps}^m)\|_{L^2}^2 + \|v_{\eps}\|_{L^{2n+2m-2}}^{2n+2m-2} \big)
        \end{split}
    \end{equation*}
    for sufficiently small $\mu>0$, constant $C_{\mu}$ depending on $\mu$ and $k$ and $\gamma(k) > 1$. The proof is completed by replacing $2n-1-k$ by $k$, changing the range of the sum, and also ``redefining" $\gamma(k)$. 
\end{proof}

\begin{cor} \label{cor:apriori}
    Let $\gamma(k)$ be the exponents in Lemma~\ref{lem:differential_inequality}. We have
    \begin{equation*}
        \|u_\eps(t)\|_{L^{2m}}^{2m} \lesssim t^{-\frac{m}{n-1}} \vee \max_{1 \leq k \leq 2n-1} \|\psi_\eps^{\diamond k}\|_{\xX_T}^{\frac{m \gamma(k)}{m+n-1}}\;,
    \end{equation*}
    where the bound does not depend on the initial condition, and
    \begin{equation*}
        \|Y_\eps(t)\|_{L^{2m}}^{2m} \lesssim \sum_{k=1}^{2n-1} \int_{0}^{t} \big\| \big( \psi_\eps(r) + e^{r \lL_\eps} u_\eps(0)  \big)^{\diamond k} \big\|_{\cC^{-\kappa}}^{\gamma(k)} {\rm d}r\;.
    \end{equation*}
    Both proportionality constants are independent of $t$. 
\end{cor}
\begin{proof}
    The first one follows from applying Lemma~\ref{lem:differential_inequality} with $v_\eps = u_\eps$ and $f_\eps  = 0$ and then \cite[Lemma~3.8]{Weber2018}. The second one is a direct consequence of Lemma~\ref{lem:differential_inequality} with $v_\eps=Y_\eps$ and $f_\eps = u_\eps(0)$ and noting that $Y_\eps(0) = 0$ in this case. 
\end{proof}

\begin{proof} [Proof of \eqref{eq:bound_u} and \eqref{eq:bound_Y}]
    The bounds \eqref{eq:bound_u} and \eqref{eq:bound_Y} follow directly from taking expectation in the bounds in Corollary~\ref{cor:apriori} and using Theorem~\ref{thm:linear_stationary}. 
\end{proof}

\subsection{Long time probabilistic bounds}

\begin{lem} \label{le: global_energy_estimate}
    For every $p \geq 1$, we have
    \begin{equation*}
        \sup_{\eps \in [0,1]} \sup_{t \in \R^+} \E \|Y_\eps(t)\|_{L^p}^p <+\infty\;.
    \end{equation*}
\end{lem}
\begin{proof}
    It suffices to consider $p=2m$ for arbitrary integer $m$. We first derive a bound for $u_\eps (t)$ for $t \geq 1$. Taking $v_\eps = u_\eps$ in Lemma~\ref{lem:differential_inequality} (and hence $f_\eps = 0$), we get
    \begin{equation*}
        \frac{{\rm d}}{{\rm d}t} \|u_\eps(t)\|_{L^{2m}}^{2m} + \|u_\eps(t)\|_{L^{2n+2m-2}}^{2n+2m-2} \lesssim \sum_{k=1}^{2n-1} \|\psi_{\eps}^{\diamond k}(t)\|_{\cC^{-\kappa}}^{\gamma (k)}\;.
    \end{equation*}
    Taking expectation on both sides, using Jensen's inequality for the second term on the left hand side, and noting $\psi_\eps$ is stationary in time, we get
    \begin{equation*}
        \frac{{\rm d}}{{\rm d}t} \big( \E \|u_\eps(t)\|_{L^{2m}}^{2m} \big) + \big(\E\|u_\eps(t)\|_{L^{2m}}^{2m} \big)^{1 +\frac{n-1}{m}} \leq \Lambda
    \end{equation*}
    for some constant $\Lambda > 0$ depending on $\kappa$, $m$ and $n$, but independent of $t \in \R^{+}$. By \cite[Lemma~3.8]{Weber2018}, we get
    \begin{equation} \label{eq:global_u}
        \E \|u_\eps(t)\|_{L^{2m}}^{2m} \lesssim t^{-\frac{m}{n-1}} \vee 1\;,
    \end{equation}
    where the proportionality constant depends on $m$, $n$ and $\kappa$, but is independent of $t$, $\eps$, and the initial condition $u_\eps(0)$. 

    As for $Y_\eps$, we have
    \begin{equation*}
        \|Y_\eps(t)\|_{L^{2m}} = \|u_\eps(t) - e^{t \lL_\eps} u_\eps(0)\|_{L^{2m}} \lesssim \|u_\eps(t)\|_{L^{2m}} + t^{-3\kappa} \|u_\eps(0)\|_{\cC^{-2\kappa}}\;.
    \end{equation*}
    Combining with \eqref{eq:global_u}, we obtain
    \begin{equation*}
        \sup_{\eps \in [0,1]} \sup_{t \geq 1} \E \|Y_\eps(t)\|_{L^{2m}}^{2m} < +\infty\;.
    \end{equation*}
    The bound for $\E \|Y_\eps(t)\|_{L^{2m}}^{2m}$ for $t \in [0,1]$ follows from Corollary~\ref{cor:apriori}. The proof is then complete. 
\end{proof}

In fact, what we get in the second part of the proof can be $\E\underset{t\in[0,1]}{\sup}||Y_{\eps}(t)||_{L^{2m}(\dD)}^{2m}<\infty$, $\forall m\in \mathbb{N}$.

\begin{lem} \label{le:global_bound_derivative}
For every $p\geq 2$, we have
\begin{equation}\label{eq: Uniform bound for Y eps}
    \begin{split}
        \sup_{\eps \in (0,1)} \sup_{t \in \R^+} \E \| Y_{\eps}(t)\|_{W^{1,+\infty}(D)}^p < \infty\;.
    \end{split}
\end{equation}
\end{lem}
\begin{proof}
    By Duhamel's formula, we have
    \begin{equation*}
        \|\nabla Y_\eps(t)\|_{L^\infty} \lesssim \sum \int_{0}^{t} \Big\| \nabla e^{(t-r) \lL_\eps} \Big( \psi_\eps^{\diamond k_1}(r) \big( e^{r \lL_\eps} \psi_\eps(0) \big)^{k_2} \big( e^{r \lL_\eps} \phi_\eps(0) \big)^{k_3} Y_{\eps}^{k_4}(r) \Big)  \Big\|_{L^\infty} {\rm d} r\;,
    \end{equation*}
    where the sum is taken over all non-negative integers $k_1, \dots, k_4$ such that $k_1 + k_2 + k_3 + k_4 = 2n-1$. Using semi-group estimates, we can further bound it by
    \begin{equation*}
        \|\nabla Y_\eps(t)\|_{L^\infty} \lesssim \sum \|\psi_{\eps}(0)\|_{\cC^{-\kappa}}^{k_2} \|\phi_\eps(0)\|_{\cC^{-\beta}}^{k_3} \int_{0}^{t} f_{k_2, k_3}(t,r) \|\psi_{\eps}^{\diamond k_1}(r)\|_{\cC^{-\kappa}} \|Y_\eps(r)\|_{\cC^{2\kappa}}^{k_4} {\rm d}r\;,
    \end{equation*}
    where
    \begin{equation} \label{eq:f_kernel}
        f_{k_2, k_3}(t,r) = (t-r)^{-\frac{1+\kappa}{2}} r^{-\frac{3\kappa k_2}{2} - \frac{(2\kappa+\beta) k_3}{2}} e^{-c ((t-r) + (k_2 +k_3) r)}
    \end{equation}
    for some $c>0$. Taking $L_{\omega}^p$-norm on both sides above and using H\"{o}lder's inequality, we get
    \begin{equation*}
        \begin{split}
        \|\nabla Y_\eps(t)\|_{L_{\omega}^{p} L_{x}^{\infty}} \lesssim &  \sum \bigg[ \|\psi_{\eps}(0)\|_{L_{\omega}^{k_2 \theta} \cC_x^{-\kappa}}^{k_2} \|\phi_\eps(0)\|_{L_{\omega}^{k_3 \theta} \cC_{x}^{-\beta}}^{k_3}\\
        &\int_{0}^{t} f_{k_2, k_3}(t,r) \|\psi_{\eps}^{\diamond k_1}(r)\|_{L_{\omega}^{\theta} \cC_{x}^{-\kappa}} \|Y_\eps(r)\|_{L_{\omega}^{k_4 \theta} \cC_{x}^{2\kappa}}^{k_4} {\rm d}r \bigg]\;,
        \end{split}
    \end{equation*}
    where $\theta = 4p$, and we recall the sum is taken over non-negative integers $k_1, \dots, k_4$ that sum up to $2n-1$. Since $\|f_{k_2, k_3}(t,\cdot)\|_{L^{1}_{[0,t]}}$ is bounded by a universal constant independent of $t$, taking supremum over $t \in [0,T]$ on both sides, and noting that $\psi_{\eps}^{\diamond k_1}$ is stationary and the pure stochastic terms and initial data have arbitrarily high moments, we get
    \begin{equation*}
        \|\nabla Y_\eps\|_{L_{T}^{\infty} L_{\omega}^{p} L_{x}^{\infty}} \lesssim \sum_{k=0}^{2n-1} \|Y_{\eps}\|_{L_{T}^{\infty} L_{\omega}^{k \theta} \cC_{x}^{2\kappa}}^{k} \lesssim 1 + \sum_{k=1}^{2n-1} \|Y_{\eps}\|_{L_{T}^{\infty} L_{\omega}^{k \theta} \cC_{x}^{2\kappa}}^{k}\;, 
    \end{equation*}
    where the proportionality constant is independent of $T$. Hence, the above bound can be replaced by taking supremum over $t \in \R^+$. By Sobolev embedding in dimension two and Proposition~\ref{prop:fractional_GN}, we have
    \begin{equation*}
        \|Y_\eps(t)\|_{\cC^{2\kappa}} \lesssim \|Y_\eps(t)\|_{W^{3\kappa, q}} \lesssim \|\nabla Y_\eps(t)\|_{L^{q}}^{3\kappa} \; \|Y_\eps(t)\|_{L^{q}}^{1-3\kappa}\;,
    \end{equation*}
    where $q < +\infty$ depends on $\kappa$. Note that we do not have the extra $L^q$ term here since $Y_\eps(t) = 0$ on $\d \dD$. Plugging this back into the above term and using H\"older's inequality, we get
    \begin{equation*}
        \|\nabla Y_\eps\|_{L_{t}^{\infty} L_{\omega}^{p} L_{x}^{\infty}} \lesssim 1 + \sum_{k=1}^{2n-1} \|\nabla Y_\eps\|_{L_{t}^{\infty} L_{\omega}^{p} L_{x}^{q}}^{12 \kappa k p} \|Y_\eps\|_{L_{t}^{\infty} L_{\omega}^{\gamma} L_{x}^{q}}^{\gamma'} \lesssim 1 + \sum_{k=1}^{2n-1} \|\nabla Y_\eps\|_{L_{t}^{\infty} L_{\omega}^{p} L_{x}^{q}}^{12 \kappa k p}\;,
    \end{equation*}
    where in the last inequality we have used the $L_{t}^{\infty} L_{\omega}^{q} L_{x}^{q}$ bound for $Y_\eps$ in Lemma~\ref{le: global_energy_estimate}, and $\gamma$ and $\gamma'$ in the middle term are some positive exponents depending on $k$, $\kappa$ and $p$, and $L_{t}^{\infty}$ is taken over $t \in \R^+$. The conclusion then follows since $\kappa$ can be arbitrarily small and hence the exponent $12 \kappa k p$ is less than $1$. 
\end{proof}

\section{Convergence}
\label{sec:convergence}

In this section, we prove convergence of the remainder $u_\eps = \phi_\eps - \psi_\eps$ to its homogenisation limit $u_0$. This, when combined with the convergence of the linear part $\psi_\eps$, implies convergence of $\phi_\eps$ to its homogenisation limit $\phi_0$\footnote{Since the solution theory to the equation \eqref{eq:phi_limit} for $\phi_0 = \psi_0 + u_0$ is by now standard, we will not explain it in details here, but refer the readers to \cite{DaPratoDebussche2003} for more details.}. 

Recall the remainder $u_\eps$ has the form
\begin{equation*}
    u_\eps(t) = e^{t \lL_\eps} \big( \phi_\eps(0) - \psi_\eps(0) \big) + Y_\eps(t)\;,
\end{equation*}
where $Y_\eps$ solves the equation
\begin{equation*}
    \d_t Y_\eps = \lL_\eps Y_\eps - \big( \psi_\eps + e^{t \lL_\eps} (\phi_\eps(0) - \psi_\eps(0)) + Y_\eps \big)^{\diamond (2n-1)}\;, \quad Y_\eps(0) = 0\;.
\end{equation*}
Let $u(0) = \phi(0) - \psi(0)$ and $u(t) = e^{t \lL_0} u(0)$. Let $Y_0$ be the solution to the equation for $Y_\eps$ with $\lL_\eps$ and $\psi_\eps$ replaced by their homogenised limits $\lL_0$ and $\psi_0$. 

For convergence of $u_\eps$ to $u_0$, we split it by
\begin{equation*}
    \|u_\eps(t) - u(t)\|_{\cC^{1-\kappa}} \leq \| (e^{t \lL_\eps} - e^{t \lL_0}) u_\eps(0) \|_{\cC^{1-\kappa}} + \|e^{t \lL_0} (u_\eps(0) - u(0))\|_{\cC^{1-\kappa}}\;.
\end{equation*}
By assumption on convergence of $\phi_\eps(0)$ to $\phi_0(0)$ and Theorem~\ref{thm:linear_stationary}, we have
\begin{equation*}
    \E \Big(\sup_{t \in [0,T]} t^{\frac{1+\beta+\kappa}{2}} \|e^{t \lL_0} (u_\eps(0) - u_0(0))\|_{\cC^{1-\kappa}} \Big)^p \rightarrow 0\;.
\end{equation*}
For the first term, interpolating \eqref{eq:semi_group_derivative} and the bound for $e^{t \lL_\eps} - e^{t \lL_0}$, we also have
\begin{equation*}
    \E \Big(\sup_{t \in [0,T]} t^{\frac{1+\beta+\kappa}{2}} \|(e^{t \lL_\eps} - e^{t \lL_0}) u_\eps(0)\|_{\cC^{1-\kappa}} \Big)^p \rightarrow 0\;.
\end{equation*}
This shows that
\begin{equation*}
    \E \Big(\sup_{t \in [0,T]} t^{\frac{1+\beta+\kappa}{2}} \|u_\eps(t) - u_0(t)\|_{\cC^{1-\kappa}} \Big)^p \rightarrow 0\;. 
\end{equation*}
It remains to consider $Y_\eps$. We have the following theorem. 
\begin{thm}
    For every $T>0$, every $s \in (0,1)$ and every $p \geq 1$, we have
    \begin{equation*}
        \E \|Y_{\eps} - Y_0\|_{L^{\infty}([0,T]; W^{s,p})}^{p} \rightarrow 0
    \end{equation*}
    as $\eps \rightarrow 0$. 
\end{thm}
\begin{proof}
    The proof consists of three steps. We first use compactness to obtain almost sure convergence of subsequences of $\{Y_\eps\}$ in $L^{\infty}([0,T]; L^p)$. We then show that any subsequential limit must satisfy the same homogenised equation, which has a unique solution. This proves the almost sure convergence of the whole sequence $Y_\eps$ in $L^{\infty}([0,T]; L^p)$. Finally, we use the uniform $W^{1,p}$ bounds in the previous section to enhance it to convergence in $L_{\omega}^{p} L_{t}^{\infty} W^{s,p}$ for arbitrary $s<1$ and $p \geq 1$. 

    \begin{flushleft}
    \textit{Step 1.}
    \end{flushleft}
    Let $\{\eps_{\ell}\}$ be an arbitrary sequence with $\eps_{\ell}\to 0$, and $p \geq 1$ be arbitrary. There exists $\Omega' \subset \Omega$ with full measure and a further subsequence, still denoted by $\{\eps_\ell\}$ such that
    \begin{equation*}
        \sum_{k=0}^{2n-1} \|\psi_{\eps_\ell}^{\diamond k} - \psi_{\eps_{\ell+1}}^{\diamond k}\|_{L_{\omega}^{p} L_{[0,T]}^{\infty} \cC^{-\kappa}} \leq 2^{-\ell}\;.
    \end{equation*}
    for every $\ell \geq 0$. Hence, we have
    \begin{equation*}
        F := \sup_{\ell \geq 0} \sum_{k=0}^{2n-1} \sup_{t \in [0,T]} \|\psi_{\eps_\ell}^{\diamond k}(t)\|_{\cC^{-\kappa}} < +\infty
    \end{equation*}
    almost surely. Furthermore, one has $\|F\|_{L_{\omega}^{p}} < +\infty$. 

    Same as in the proof of Lemma~\ref{le:global_bound_derivative}, by Duhamel formula and properties of the semi-group $\nabla e^{(t-r) \lL_\eps}$, we have
    \begin{equation*}
        \|\nabla Y_\eps(t)\|_{L^\infty} \lesssim \sum \|\psi_{\eps}(0)\|_{\cC^{-\kappa}}^{k_2} \|\phi_\eps(0)\|_{\cC^{-\beta}}^{k_3} \int_{0}^{t} f_{k_2, k_3}(t,r) \|\psi_{\eps}^{\diamond k_1}(r)\|_{\cC^{-\kappa}} \|Y_\eps(r)\|_{\cC^{2\kappa}}^{k_4} {\rm d}r\;,
    \end{equation*}
    where the sum is taken over $k_1 + \cdots + k_4 = 2n-1$ and the kernel $f_{k_2, k_3}$ is given by \eqref{eq:f_kernel}. Taking $L_{t}^{\infty}([0,T])$ on both sides and using integrability of the kernel $f(t,\cdot)$, we get
    \begin{equation*}
        \|\nabla Y_\eps\|_{L_{T}^{\infty} L_{x}^{\infty}} \lesssim \sum \|\psi_\eps(0)\|_{\cC^{-\kappa}}^{k_2} \|\phi_\eps(0)\|_{\cC^{-\beta}}^{k_3} \|\psi_{\eps}^{\diamond k_1}\|_{L_{T}^{\infty} \cC_{x}^{-\kappa}} \|Y_\eps\|_{L_{T}^{\infty} \cC_{x}^{2\kappa}}^{k_4}\;.
    \end{equation*}
    Again, by Sobolev embedding in dimension two and Proposition~\ref{prop:fractional_GN}, we have
    \begin{equation*}
        \|Y_{\eps}(t)\|_{\cC^{2\kappa}} \lesssim \|Y_{\eps}(t)\|_{W^{3\kappa, q}} \lesssim \|\nabla Y_{\eps} ( t )\|_{L^{q}}^{3\kappa} \; \|Y_{ \eps}(t)\|_{L^{q}}^{1-3\kappa}\;,
    \end{equation*}
    where $q < +\infty$ depends on $\kappa$. Plugging this back into the above term, then using H\"older's inequality as well as the almost sure finiteness of $\|\psi_{\eps_\ell}(0)\|_{\cC^{-\kappa}}$ and $\|\psi_{\eps_\ell}^{\diamond k_1}\|_{L_{T}^{\infty} \cC_{x}^{-\kappa}}$ uniformly along the subsequence $\{\eps_\ell\}$, we get
    \begin{equation*}
        \underset{t\in [0,T]}{\sup}\|\nabla Y_{\eps_{\ell}}(t)\|_{L_{x}^{\infty}} \lesssim_{\omega, T} 1 + \sum_{k=1}^{2n-1} \|\nabla Y_{\eps_{\ell}}\|_{L_{T}^{\infty}  L_{x}^{q}}^{12 \kappa k p} \|Y_{\eps_{\ell}}\|_{L_{T}^{\infty}  L_{x}^{q}}^{\gamma'} \lesssim_{\omega, T} 1 + \sum_{k=1}^{2n-1} \|\nabla Y_{\eps_{\ell}}\|_{L_{T}^{\infty}  L_{x}^{q}}^{12 \kappa k p}\;,
    \end{equation*}
    where the proportionality constant is now random but almost surely finite and independent of $\eps_\ell$. This then implies
    \begin{equation}\label{eq: uniform estimate for W 1 infty norm}
    \underset{\ell \in \NN}{\sup}\underset{t\in[0,T]}{\sup}|| Y_{\eps_{\ell}}(t)||_{W_0^{1,\infty}(\dD)}<\infty
    \end{equation}
    almost surely. 

    On the other hand, since
    \begin{equation*}
    \partial_{t} Y_{\eps_{\ell}}=\lL_{\eps_{\ell}}Y_{\eps_{\ell}} - \big(\psi_{\eps_{\ell}} + e^{t \lL_{\eps_{\ell}} } u_{\eps_{\ell}}(0) + Y_{\eps_{\ell}} \big)^{\diamond  (2n - 1)}\;,
    \end{equation*}
    the previous bounds on $Y_{\eps_\ell}$ then implies
    \begin{equation} \label{eq: uniform estimate for W -1 infty norm}
    \underset{\ell \in \NN}{\sup} \| \partial_t Y_{\eps_{m_l}} \|_{L^{4}(0,T; W^{-1,\infty}(\dD))}(\omega) < \infty 
    \end{equation}
    almost surely. Combining \eqref{eq: uniform estimate for W 1 infty norm}, \eqref{eq: uniform estimate for W -1 infty norm}, and applying Arzela-Ascoli theorem, we conclude that there exists a further subsequence (which we still denoted by $\{Y_{\eps_\ell}\}$) that that converges almost surely to some $\bar{Y_0}$, in $L^{\infty}([0,T]; L^p)$ for every $p \geq 1$. 

    \begin{flushleft}
    \textit{Step 2.}
    \end{flushleft}
    
    We now show that the almost sure subsequential limit $\bar{Y_0}$ must satisfy the mild form equation
    \begin{equation} \label{eq:homogenised_limit_remainder}
        \bar{Y_0}(t) = -\int_{0}^{t} e^{(t-r) \lL_0} \big( \psi_0(r) + e^{r \lL_0}(\phi_0(0) - \psi_0(0)) + \bar{Y}_0(r) \big)^{\diamond (2n-1)} {\rm d}r\;,
    \end{equation}
    and then the convergence of the whole sequence to this limit follows from uniqueness of the solution to this equation. For convenience, we write
    \begin{equation*}
        R_{\eps}(r) = \big( \psi_\eps(r) + e^{r \lL_\eps}(\phi_\eps(0) - \psi_\eps(0)) + Y_\eps(r) \big)^{\diamond (2n-1)}\;,
    \end{equation*}
    and $\bar{R_0}(r)$ by setting the corresponding $\eps$ to $0$ and $Y_\eps$ to $\bar{Y_0}$. Note that $Y_\eps$ satisfies the mild form equation
    \begin{equation} \label{eq:Duhamel_remainder}
        Y_\eps(t) = - \int_{0}^{t} e^{(t-r) \lL_\eps} R_\eps(r) {\rm d}r\;.
    \end{equation}
    Let $\{Y_{\eps_\ell}\}$ be a subsequence that converges almost surely to $\bar{Y_0}$ above. In order to show $\bar{Y_0}$ satisfies \eqref{eq:homogenised_limit_remainder}, it suffices to show that the right hand side of \eqref{eq:Duhamel_remainder} converges to that of \eqref{eq:homogenised_limit_remainder} weakly (as space-time distributions) along the subsequence $\{\eps_\ell\}$; that is, for every $F \in C_{c}^{\infty}([0,T] \times \dD)$, one has
    \begin{equation} \label{eq:homogenisation_test}
        \int_{0}^{T} \Bigscal{\int_{0}^{t} e^{(t-r) \lL_{\eps_\ell}} R_{\eps_\ell}(r) {\rm d}r\;, F(t)} \, {\rm d}t \rightarrow \int_{0}^{T} \Bigscal{\int_{0}^{t} e^{(t-r) \lL_0} \bar{R_0}(r) {\rm d}r, F(t)} \, {\rm d}t\;
    \end{equation}
    as $\eps \rightarrow 0$. Here, $\scal{\cdot, \cdot}$ represents integration with respect to $x \in \dD$. For the left hand side above, changing the order of integration and then using duality, we get
    \begin{equation*}
        \int_{0}^{T} \Bigscal{\int_{0}^{t} e^{(t-r) \lL_{\eps_\ell}} R_{\eps_\ell}(r) {\rm d}r\;, F(t)} \, {\rm d}t = \int_{0}^{T} \Bigscal{R_{\eps_\ell}(r), \int_{r}^{T} e^{(t-r)\lL_{\eps_\ell}^*} F(t) {\rm d}t } \, {\rm d}r\;.
    \end{equation*}
    Since $F \in C_{c}^{\infty}([0,T] \times D)$, by standard parabolic homogenisation results (\cite{GengShen2017, Geng2020}), we have
    \begin{equation*}
        \Big\| \int_{r}^{T} \big( e^{(t-r) \lL_{\eps_\ell}^*} - e^{(t-r) \lL_0^*} \big) F(t) {\rm d}t \Big\|_{L^{\infty}([0,T]; \cC^{\gamma})} \rightarrow 0
    \end{equation*}
    for $\gamma \in (0,1)$. On the other hand, since $Y_\eps \rightarrow \bar{Y_0}$ in $L^{\infty}([0,T], L^p)$ almost surely, we have
    \begin{equation*}
        \sup_{r \in [0,T]} r^{\frac{1}{5}} \|R_\eps(r) - \bar{R_0}(r)\|_{\cC^{-\kappa}} \rightarrow 0
    \end{equation*}
    as $\eps \rightarrow 0$. This shows the convergence in \eqref{eq:homogenisation_test}, and hence the almost sure convergence of the whole sequence $Y_\eps$ to $Y_0$ in $L^{\infty}([0,T]; L^p)$. 

    \begin{flushleft}
    \textit{Step 3.}
    \end{flushleft}
    
    To show the convergence also holds in $L_{\omega}^{p} L_{t}^{\infty} W^{s,p}$ for arbitrary $s \in (0,1)$ and $p \geq 1$, it suffices to note that we have the uniform bound
    \begin{equation*}
        \sup_{\eps \in (0,1)} \|Y_\eps\|_{L_{\omega}^{q} L_{t}^{\infty} W_{x}^{1,\infty}} < +\infty
    \end{equation*}
    for arbitrary $q \geq 1$. In particular, together with almost sure convergence in $L_{t}^{\infty} L_{x}^{p}$, it implies that
    \begin{equation*}
        \E \|Y_\eps - Y_0\|_{L_{t}^{\infty} L_{x}^{p}}^{q'} \rightarrow 0
    \end{equation*}
    for every $q' < q$. The time interval is fixed to be $[0,T]$. Since we also have the interpolation
    \begin{equation*}
        \|Y_{\eps}(t) - Y_0(t)\|_{W^{s,p}} \lesssim \|Y_\eps(t) - Y_0(t)\|_{L^p}^{\theta} \|Y_\eps(t) - Y_0(t)\|_{W^{1,p}}^{1-\theta}\;
    \end{equation*}
    for some $\theta \in (0,1)$ depending on $s$ and $p$, we get convergence in $L_{\omega}^{q'} L_{t}^{\infty} W_{x}^{s,p}$. Note that $q' < q$ and $p$ are both arbitrary. This completes the proof of the theorem. 
\end{proof}

\appendix

\section{Green function and semi-group estimates}

We cite some useful Green function estimates from \cite{Geng2023GaussianBA,Geng2020} and bounds on semi-groups from \cite{Suslina2016}. 

Let $a: \R^d \rightarrow \R^{d \times d}$ be a coefficient matrix satisfying Assumption~\ref{as:a} (in general dimension $d$). Let $G_\eps(t;x,y)$ be the Green function for $\d_t - \lL_\eps$ in the sense that
\begin{equation*}
\begin{split}
\left\{
\begin{array}{rll}
        \d_t G_\eps (\cdot \,; \, \cdot,  y) &= \lL_\eps G_\eps (\cdot \,; \, \cdot, y) \quad  & \text{on} \; \; (0,+\infty)\times \mathcal{D}\;\\
             G_\eps (\cdot \,; \, \cdot, y) &= 0 \quad &\text{on} \; \; (0,+\infty) \times \partial\mathcal{D}\;\\
        \underset{t \downarrow 0^{+}}{\lim} G_\eps(t; \cdot, y) &= \delta_y (\cdot)   & 
\end{array}
\right.
\end{split}
\end{equation*}
By \cite[Theorem~2.2]{Geng2020},
\begin{equation} \label{eq:green_pointwise}
\begin{split}
    |G_\eps(t; x,y)| \lesssim \frac{1}{\big( \sqrt{t} + |x-y| \big)^{d}}\;.
\end{split}
\end{equation}
By \cite[Theorem~2.2]{Geng2020} and \cite[Theorem~1.2]{GengShen2015},
\begin{equation} \label{eq:green_gradient_pointwise}
\begin{split}
    |\nabla_x G_{\eps}(t;x,y)| + |\nabla_y G_{\eps}(t;x,y)| \lesssim \frac{1}{\big( \sqrt{t} + |x-y| \big)^{d+1}}\;
\end{split}
\end{equation}
and
\begin{equation} \label{eq:green_gradient_gradient_pointwise}
\begin{split}
    |\nabla_x \nabla_y G_{\eps}(t;x,y)| \lesssim \frac{1}{\big( \sqrt{t} + |x-y| \big)^{d+2}}\;.
\end{split}
\end{equation}
uniformly over all $t, x, y$ such that $|t| < \diam(\dD)$. By \cite[Theorem~1.1]{Geng2023GaussianBA}, 
\begin{align} \label{eq:green_convergence}
    |G_{\varepsilon}(t;x,y)-G_{0}(t;x,y)|\lesssim \frac{\varepsilon^\frac{1}{2}}{(|x-y|+\sqrt{|t|})^{d+1}}.
\end{align}

Let $e^{t\mathcal{L}_{\eps}}$ be the semi-group generated by a linear closed, and densely defined operator $\lL_{\eps}$ on $L^{2}(\dD)$. Let $e^{t\mathcal{L}_{0}}$ be the semi-group generated by a linear closed, and densely defined operator $\lL_{0}$ on $L^{2}(\dD)$. By \cite[Lemma~4.1]{Suslina2016},
\begin{equation}\label{eq:semi_group_operator_bound}
\|e^{t\mathcal{L}_{\eps}} \|_{L^{2}( \dD ) \to L^2( \dD ) } \lesssim e^{-Ct}\;, \quad \|e^{t\mathcal{L}_{\eps}}\|_{L^{2}( \dD ) \to H^1( \dD ) } \lesssim t^{-1/2}e^{-Ct}\;.
\end{equation}
By \cite[Theorem~4.2]{Suslina2016},
\begin{equation} \label{eq:semi_group_difference}
\begin{split}
\|e^{t\mathcal{L}_{\eps}}-e^{t\mathcal{L}_{0}}\|_{L^{2}( \dD ) \to L^2( \dD ) }\lesssim  \frac{\eps}{(t+\eps^2)^{1/2}}e^{-Ct}.
\end{split}
\end{equation}
 By \eqref{eq:green_pointwise}, \eqref{eq:green_convergence} and \eqref{eq:semi_group_operator_bound}, $\forall\alpha,\beta\in (0,1)$, $\forall\kappa\in(0,\frac{1}{10})$,
\begin{equation} \label{eq:semi_group_derivative}
    \|e^{t \lL_\eps} - e^{t \lL_0}\|_{\cC^{\alpha}(\dD)} \lesssim_{\alpha,\beta,\kappa} \varepsilon^{\kappa}t^{-\frac{\alpha+\beta+5\kappa}{2}} e^{-C t} \|f\|_{\cC^{-\beta}(\dD)}\;.
\end{equation}

\section{Some functional inequalities}

We fix the bounded domain $\dD \subset \R^d$ and $\alpha \in (0,1)$ throughout this section. All proportionality constants depend on $\dD$, the value of $\alpha$ and any differentiability or integrability exponents in the statements. 

\begin{prop} \label{prop:B11_bound}
    $\forall\alpha\in(0,1)$, $\forall f\in B^{\alpha}_{1,1}$,
    \begin{equation*}
        \|f\|_{B_{1,1}^{\alpha}} \lesssim_{\alpha}\|\nabla f\|_{L^1}^{\alpha} \|f\|_{L^1}^{1-\alpha}+\|f\|_{L^{1}}\;.
    \end{equation*}
\end{prop}
\begin{proof}
    This is \cite[Proposition~8]{Weber2017} for $\mu=0$. 
\end{proof}

\begin{prop} [Fractional Gagliardo-Nirenberg inequality]
\label{prop:fractional_GN}
     For every $q \in (1,+\infty)$, $\alpha\in(0,1)$, we have
    \begin{equation*}
        \|f\|_{W^{\alpha,q}} \lesssim_{\alpha,q} \|f\|_{W^{1,q}}^{\alpha} \|f\|_{L^q}^{1-\alpha}
    \end{equation*}
\end{prop}
\begin{proof}
    This is \cite[Theorem~1]{fractional_GN_inequality}. 
\end{proof}

\textsc{School of Mathematical Sciences, Peking University, 5 Yiheyuan Road, Haidian District, Beijing, 100871, China}. 
\\
Email: 2001110006@stu.pku.edu.cn
\\
\\
\textsc{Beijing International Center for Mathematical Research, Peking University, 5 Yiheyuan Road, Haidian District, Beijing, 100871, China}. 
\\
Email: weijunxu@bicmr.pku.edu.cn

\end{document}